\documentclass{amsart}
\usepackage{amscd}
\usepackage{amsmath, amssymb}
\usepackage{amsfonts}
\usepackage{enumerate}

\newcommand{\bc}{\begin{center}}
\newcommand{\ec}{\end{center}}
\newcommand{\ba}{\begin{array}}
\newcommand{\ea}{\end{array}}
\renewcommand{\a}{\alpha }
\renewcommand{\b}{\beta }

\renewcommand{\d}{\delta }

\renewcommand{\l}{\lambda }

\renewcommand{\dfrac}{\displaystyle\frac }

\numberwithin{equation}{section}

\theoremstyle{plain}
\newtheorem{thm}{Theorem}
\newtheorem{defn}{Definition}
\newtheorem{lem}{Lemma}
\newtheorem{rem}{Remark}

\newtheorem{cor}{Corollary}
\newtheorem{proposition}{Proposition}

\def\ZZ{{\mathbb Z}}

\def\RR{{\mathbb R}}
\def\CC{{\mathbb C}}

\def\begeq{\begin{equation}}
\def\endeq{\end{equation}}

\begin{document}
\title[Uniqueness of the Mean Field Equation and Rigidity of Hawking Mass]{Uniqueness of the Mean Field Equation and Rigidity of Hawking Mass}
\author{Yuguang Shi, Jiacheng Sun, Gang Tian, Dongyi Wei}
\address{School of Mathematical Sciences and BICMR, Peking University, Yiheyuan Road 5, Beijing, P.R.China, 100871}
\email{ygshi@math.pku.edu.cn, sunxujason@pku.edu.cn, tian@math.princeton.edu, jnwdyi@pku.edu.cn}
\maketitle
\begin{abstract}
In this paper, we prove that the even solution of the mean field equation $\Delta u=\lambda(1-e^u) $ on $S^2$ must be axially symmetric when $4<\lambda \leq 8$. In particular, zero is the only even solution for $\lambda=6$. This implies the rigidity of Hawking mass for stable constant mean curvature(CMC) sphere with even symmetry.\end{abstract}
\section{Introduction}
Consider the mean field equation
\begin{equation}\label{eq:1}
\frac{\a}{2}\Delta u+\frac{e^u}{\int_{S^2}e^ud\omega}-1=0,\ \ \text{on}\ \ S^2,
\end{equation}
where $ \Delta$ denotes the Laplace Beltrami operator with respect to $(S^2,g_0)$, $(S^2,g_0)$ is the unit sphere of $\RR^3$ equipped with the metric
$g_0$ induced from the flat metric of $\RR^3$, the volume form $d\omega$ is normalized so that $\int_{S^2}d\omega=1.$ Clearly, equation \eqref{eq:1} is the Euler-Lagrange equation of the nonlinear functional on the Sobolev space $H^1(S^2)$
$$J_{\a}(u)=\frac{\a}{4}\int_{S^2}e^ud\omega+\int_{S^2}ud\omega-\ln\int_{S^2}e^ud\omega.$$  Since \eqref{eq:1} is invariant under
adding a constant, we can normalize $\int_{S^2}e^ud\omega=1, $ then \eqref{eq:1} becomes
\begin{equation}\label{eq:2}
\frac{\a}{2}\Delta u+{e^u}-1=0,\ \ \text{on}\ \ S^2.
\end{equation}
Recently, Gui and Moradifam \cite{Gui} proved that for $\a\geq\frac{1}{2}$, any solution of \eqref{eq:2} in
$$\mathcal{M}=\{u\in H^1(S^2):\int_{S^2}e^ux_i=0,\ i=1,2,3\}$$ must be $0$. This implies the Moser-Trudinger inequality $$\inf_{u\in\mathcal{M}}J_{\a}(u)=0,\ \ \text{for}\ \ \a\geq\frac{1}{2}.$$
Inspired by the above result, we study the case $\a\geq\frac{1}{4}$ with even symmetry, i.e. $u(x)=u(-x),$ and our result is

\begin{thm}\label{1.1}
Let $\frac{1}{4}\leq\a<1$ and assume $u$ a solution of \eqref{eq:2},  $u(x)=u(-x),$ then $u$ must be
axially symmetric, moreover if $\a=\frac{1}{3}$ then $u=0.$
\end{thm}

The proof of Theorem \ref{1.1} is based on the Sphere Covering Inequality and the nodal set
analysis. The Sphere Covering Inequality in \cite{Gui} is stated as follows:

\begin{lem}\label{1.2}Let $\Omega$ be a simply-connected subset of $\RR^2$ and assume $w_i\in C^2(\overline{\Omega})$, i=1,2 satisfy $$\triangle w_i+e^{w_i}=f_i(y),$$ where $f_2\geq f_1\geq0$ in $\Omega$, if $w_2> w_1$ in $\Omega$ and $w_2= w_1$ on $\partial\Omega$, then $$\int_{\Omega} (e^{w_1}+e^{w_2})dy\geq 8\pi.$$ Moreover, if the equality holds then $f_1\equiv f_2\equiv0.$
\end{lem}
The following is a consequence of Bol's inequality and the equimeasurable symmetric
rearrangement (see Lemma 3.1 in \cite{Lin1}, Proposition 3.2 in \cite{Gui}, Proposition 3.3 in \cite{Lin2}) and can be regarded
as a limiting case of the Sphere Covering Inequality.
\begin{lem}\label{1.3} Let $\Omega$ be a simply-connected subset of $\RR^2$ and assume that
$w\in C^2(\overline{\Omega})$ satisfies $\triangle w+e^{w}>0,$ in $\overline{\Omega} $ and
$\int_{\Omega} e^{w}dy\leq 8\pi$ in $\Omega$. Consider an open set $\omega\subset\Omega$ and define $$\l_{1,w}(\omega)=\inf_{\phi\in H_0^1(\omega),\|\phi\|_{L^2}=1} \left(\int_{\omega}|\nabla\phi|^2-\int_{\omega}|\phi|^2e^{w}\right)\leq 0.$$ Then $\int_{\omega} e^{w}dy> 4\pi.$
\end{lem}

One of our motivation for proving Theorem 1 is to study the rigidity problem on Hawking mass for surfaces in three manifolds with nonnegative scalar curvature. The Hawking mass is a quasi-local mass in mathematical theory of general relativity. It has played a very important role in proving the Penrose Inequality in both \cite{Br} and \cite{HI}. In \cite{CY}. Christodoulou and Yau prove that for a stable $CMC$ sphere in nonnegative scalar curvature three manifold, the Hawking mass is nonnegative.

During his talk at ICM 2002 in Beijing, Bartnik proposed the rigidity problem on Hawking mass (see \cite[P.235]{Bar}).
In \cite{Sun}, the second author proves the following:
\begin{lem}\cite{Sun}
\label{lem:sun1}
Let (M,g) be a complete Riemnnian three manifold with scalar curvature $R(g)\geq0$ and $\Omega \subset M$ be a domain with boundary $\Sigma=\partial\Omega$.
We further assume that $\Sigma$ is a stable CMC sphere with $m_H(\Sigma)= 0$, where $m_H(\Sigma)$ denotes the Hawking mass of $\Sigma$. Then $\Omega$ is isometric to a Euclidean ball in $\mathbb{R}^3$
whenever the mean field equation
\begin{eqnarray}
\Delta u \,= \,6\,(1-e^u),~~~~{\rm on}~~S^2\nonumber
\end{eqnarray}
has only zero solution.
\end{lem}
This reduces the rigidity problem on Hawking mass to the uniqueness of zero solution of \eqref{eq:2} with $\a=\frac{1}{3}$. 
 In \cite{Sun}, the author solved this uniqueness problem locally, that is, zero is the only solution for \eqref{eq:2} with $\a=\frac{1}{3}$ which is sufficiently small.
As a consequence, the second author proved the rigidity of Hawking mass for nearly round surface. Now by using Lemma \ref{lem:sun1} and Theorem \ref{1.1}, we can have
\begin{thm}\label{1.2}
Let (M,g) be a complete Riemnnian three manifold with scalar curvature $R(g)\geq0$ and $\Omega \subset M$ be a domain with boundary $\Sigma=\partial\Omega$, if $\Sigma$ is a stable CMC sphere with even symmetry and $m_H(\Sigma)= 0$, then $\Omega$ is isometric to a Euclidean ball in $\mathbb{R}^3$. In particular, $\Sigma$ is isometric to the standard $\mathbb{S}^2$ in $\mathbb{R}^3$.
\end{thm}
Here , by even symmetry of $\Sigma$, we mean that there exists an isometry $ \rho:\Sigma\to\Sigma$ satisfying: $ \rho^2=id$ and $ \rho(x)\neq x$ for $x\in\Sigma$. There is also a hyperbolic version of the rigidity of Hawking mass:
\begin{thm}\label{1.3}
Let (M,g) be a complete Riemnnian three manifold with scalar curvature $R(g)\geq-6$ and $\Omega \subset M$ be a domain with boundary $\Sigma=\partial\Omega$, if $\Sigma$ is a stable CMC sphere with even symmetry and $m_H(\Sigma)= 0$, then $\Omega$ isometric to a hyperbolic ball in $\mathbb{H}^3$.
\end{thm}

We can also get the flatness for isoperimetric surface of sphere type:
\begin{thm}\label{1.4}
Let (M,g) be an asymtotically flat three manifold with scalar curvature $R(g)\geq0$.
If there exists an isoperimetric sphere $\Sigma$ with even symmetry and $m_H(\Sigma)\\= 0$, then (M,g) is isometric to $(\mathbb{R}^3, \delta) $, where $\delta$ denotes the flat Euclidean metric.
\end{thm}
There is also a hyperbolic version of this theorem:
\begin{thm}\label{1.5}
Let (M,g) be an asymptotically hyperbolic three manifold with scalar curvature $R(g)\geq-6$.
If there exists an isoperimetric sphere $\Sigma$ with even symmetry and $m_H(\Sigma)= 0$, then (M,g) is isometric to the hyperbolic space form $(\mathbb{H}^3, g_{\mathbb{H}}) $.
\end{thm}\begin{rem}
Using Theorem \ref{1.1} and spectral gap we can prove Lemma \ref{3.3} in section 4, which improves the local uniqueness result in \cite{Sun}
 to a more precise form. By Lemma \ref{3.3} we know that the even symmetry assumption can be replaced by the Gaussian curvature bound
 $K_{\Sigma}{\dfrac{|\Sigma|}{4\pi}}\leq2,$  in Theorem
 \ref{1.2}-\ref{1.5}.\end{rem}

\section{Preliminaries}
We give some basic notations to present our result. Let $\Sigma \subset (M,g)$ be a surface with unit normal vector field $n$, second fundamental form $A$ and mean curvature $H$.
\begin{defn}
The Willmore functional of $\Sigma$ is defined by:
\begin{eqnarray}
W(\Sigma)=\frac{1}{4}\int_{\Sigma}{H^2}.
\end{eqnarray}
when $R(g)\geq0$.
\begin{eqnarray}
W(\Sigma)=\frac{1}{4}\int_{\Sigma}{(H^2-4)}.
\end{eqnarray}
when $R(g)\geq-6$.
\end{defn}
Willmore functional appears naturally in general relativity in form of Hawking mass of a surface:
\begin{defn}
The Hawking mass of $\Sigma$ is defined by:
\begin{eqnarray}
m_H(\Sigma)=\frac{{|\Sigma|^{\frac{1}{2}}}}{(16\pi)^{\frac{3}{2}}}(16\pi-\int_{\Sigma}{H^2}).
\end{eqnarray}
when $R(g)\geq0$.
\begin{eqnarray}
m_H(\Sigma)=\frac{{|\Sigma|^{\frac{1}{2}}}}{(16\pi)^{\frac{3}{2}}}(16\pi-\int_{\Sigma}(H^2-4)).
\end{eqnarray}
when $R(g)\geq-6$.
\end{defn}
\begin{defn}
If $H$ is constant along $\Sigma$, we say $\Sigma$ is a CMC surface;
\\The Jacobi operator of a CMC surface $\Sigma$ is the second variation of area:
\begin{eqnarray}
L_{\Sigma}=-\Delta_{\Sigma}-(|A|^2+Ric(n,n))
\end{eqnarray}

A CMC surface $\Sigma$ is $stable$ if the first eigenvalue of $L_{\Sigma}$ on mean zero functions is nonnegative
\begin{eqnarray}
\Lambda_1(L_{\Sigma})=inf\{\int_{\Sigma}fL_{\Sigma}f: \int_{\Sigma}f=0, \int_{\Sigma}f^2=1\}\geq 0
\end{eqnarray}
i.e. it satisfies the following stability condition:
\begin{eqnarray}
\int_{\Sigma}{(|A|^2+Ric(n,n))f^2}\leq \int_{\Sigma}{|\nabla f|^2}
\end{eqnarray}
for all $f\in C_c^\infty(\Sigma)$ and $\int_{\Sigma}f=0$.
\end{defn}
%
We also want to study the isoperimetric surface in $AF(resp. AH)$ three manifold, we will always use the bracket to denote asymptotic hyperbolic case after related asymptotic flat situations.
\begin{defn}
A complete connected three manifold (M,g) is called AF(resp. AH), if there exists a constant $C> 0$, a compact set $K$, such that $M\setminus K$ is diffeomorphic to $\mathbb{R}^3\setminus B_R(0)$ for some $R>0$, and in standard coordinate the metric g has the following properties:
\begin{eqnarray}
g=\delta+h(resp. g=g_{\mathbb{H}}+h)
\end{eqnarray}
and
\begin{eqnarray}
|h_{ij}|+r|\partial h_{ij}|+r^2|\partial^2 h_{ij}|\leq Cr^{-\tau}
\end{eqnarray}
$\tau\in (\frac{1}{2},1](resp. \tau=3)$, where $r$ and $\partial$ denote the Euclidean distance and standard derivative operator on $\mathbb{R}^3$ respectively. The region $M\setminus K$ is called the end of M.
\\The standard hyperbolic space $(\mathbb{H}^3, g_{\mathbb{H}}) $ is
\begin{eqnarray}
g_{\mathbb{H}}=\frac{1}{1+r^2}dr^2+r^2g_{\mathbb{S}^2}
\end{eqnarray}
\end{defn}
We also need the following definition of isoperimetric surface.
\begin{defn}
Given a complete Riemannian 3-manifold (M,g), its isoperimetric profile with volume $V$ is defined as
\begin{eqnarray}
 \quad\quad\quad I(V) = inf\{ \mathcal{H}^2(\partial^* \Omega) : \Omega \subset M \ is \ a\ Borel\ set\ with \\ \nonumber\ finite \ perimeter\ and\ \mathcal{H}_g^3(\Omega)=V\}.
\end{eqnarray}
\end{defn}
Where, $\mathcal{H}^2$ is a 2-dim Hausdorff measure for the reduced boundary of $\Omega$ denoted by $\partial^* \Omega$. A Borel set $\Omega \subset M$ of finite perimeter such that $\mathcal{H}_g^3(\Omega)=V$ and $I(V)=\mathcal{H}^2(\partial^* \Omega)$ is called an isoperimetric region of $(M, g)$ of volume V. The surface $\partial \Omega$ is called isoperimetric surface.

\section{Integrability conditions for \eqref{eq:2}}
Assume $\a>0$ and let $\l=2/\a,$ then \eqref{eq:2} is equivalent to
\begin{eqnarray}\label{lambda equation}
\Delta u = \lambda(1-e^u)\ \ \text{on}\ \ \ S^2.
\end{eqnarray}
The mean field equation has been studied in various aspects, 
such as prescribed Gaussian curvature \cite{KW}, mean field model, Chern Simons Higgs model. This kind of equation may have bifurcation when approach $\lambda= 2k, k\in N$, so it may lose compactness. Ding, Jost, Li, Wang \cite{DJLW}\cite{DJLW1} have studied the equation at the first eigenvalue, Li \cite{Li} has initiated study of the existence of solutions by computing the Leray-Schauder topological degree, Lin computed the degree on $S^2$ in \cite{Lin} and surface of any genus \cite{CL}.

We have some integrability conditions for this equation from Kazdan-Warner \cite{KW}, this can help us to understand equation (\ref{lambda equation}). On any two dimensional manifold $M$, we have this equality:

\begin{lem}
For any smooth functions $u$ and $F$ on $M$, we have
\begin{eqnarray}\label{int condition}
2\Delta u(\nabla F\cdot \nabla u)=\nabla(2(\nabla F\cdot \nabla u)\nabla u-|\nabla u|^2\nabla F)\nonumber\\
-(2\nabla^2 F-(\Delta F)g)(\nabla u,\nabla u),
\end{eqnarray}
where $g$ is the Riemannian metric on $M$, $\nabla^2 F$ is the Hessian of $F$.
\end{lem}
\begin{proof}
We only need to prove
\begin{eqnarray}
2\nabla(\nabla F\cdot \nabla u)\nabla u=\nabla(|\nabla u|^2\nabla F)+(2\nabla^2 F-(\Delta F)g)(\nabla u,\nabla u).
\end{eqnarray}
In fact,
\begin{eqnarray}
2\nabla(\nabla F\cdot \nabla u)\nabla u &=& 2\nabla^2 F(\nabla u,\nabla u)+2\nabla^2 u(\nabla u,\nabla F)\nonumber\\
&=&2\nabla^2 F(\nabla u,\nabla u)+\nabla(|\nabla u|^2\nabla F)-(\Delta F)|\nabla u|^2.
\end{eqnarray}
\end{proof}
Now we consider the equation on standard $S^2$, take $F$ as the first order spherical harmonics on $S^2$, then we have
\begin{lem}\label{4.2}
If $u$ satisfies
\begin{eqnarray}
\Delta u = \lambda(1-e^u)\nonumber
\end{eqnarray}
on standard $S^2$, then
\begin{eqnarray}\label{zero center}
(\lambda-2)\int_{S^2}{x_ie^u}=0.
\end{eqnarray}
\end{lem}
\begin{proof}
Take $F=x_i,i=1,2,3$ in formula (\ref{int condition}), then we have

\begin{eqnarray}
-(\Delta F)=2F, \quad 2\nabla^2 F=(\Delta F)g.
\end{eqnarray}
Integrating (\ref{int condition}) on $S^2$,
\begin{eqnarray}
0=\int_{S^2}{\Delta u(\nabla x_i\cdot \nabla u)}=\lambda \int_{S^2}{(1-e^u)(\nabla x_i\cdot \nabla u)}.
\end{eqnarray}
So we get

\begin{eqnarray}
\int_{S^2}{\nabla x_i\cdot \nabla u}=\int_{S^2}\nabla x_i\cdot \nabla e^u.
\end{eqnarray}
Integrating by part on both sides,
\begin{eqnarray}
\lambda \int_{S^2}{x_i(1-e^u)}=\int_{S^2}{x_i\Delta u}=\int_{S^2} \Delta x_i e^u=-2\int_{S^2} x_i e^u
\end{eqnarray}
Since the integration of coordinate functions on $S^2$ vanishes, so we get (\ref{zero center})¡£
\end{proof}
We know from the above lemma that 
solutions of equation (\ref{lambda equation}) must lie in $\mathcal{M}$ except $\lambda=2$. In particular, for $\lambda=6$, solution of equation (\ref{lambda equation}) lies in $\mathcal{M}$. For $\lambda=2$, Onofri \cite{Ono} proved that the only solution of equation (\ref{lambda equation}) in $\mathcal{M}$ is zero. From Lemma \ref{4.2} and the result of Gui and Moradifam \cite{Gui}, we know that for $0<\l\leq4,\ \l\neq2,$ the solution of equation (\ref{lambda equation}) must be zero. Theorem \ref{1.1} tells us that for $4<\l\leq8,$ the even solution of equation (\ref{lambda equation}) must be axially symmetric. In the case $4<\l<6,$ the result of Lin \cite{Lin} tells us that the degree of \eqref{lambda equation} is 0, since the solution $u=0$ of \eqref{lambda equation} is non degenerate, it must have other solutions. Due to the possible blow up at $\l\to6,$ we conjecture that for $4<\l\leq6,$ the solution of equation (\ref{lambda equation}) must be even, thus axially symmetric. This will imply that for $\l=6,$ the solution of equation (\ref{lambda equation}) must be zero and will completely prove the rigidity of Hawking mass.

If we take $F$ as the second order spherical harmonics $\Delta F=-6F$, then we have
\begin{lem}
If $u$ satisfies
\begin{eqnarray}
\Delta u = \lambda(1-e^u)\nonumber
\end{eqnarray}
on standard $S^2$, then
\begin{eqnarray}\label{second eigenvalue}
\lambda(6-\lambda)\int_{S^2}{Fe^u}=\frac{1}{2}\int_{S^2}{(2\nabla^2 F-(\Delta F)g)(\nabla u,\nabla u)}.
\end{eqnarray}
\end{lem}
We can prove that there are no nonzero axially symmetric solutions for $\lambda=6$.
\begin{proposition}\label{axis sym}
If $u$ is an axially symmetric solution of
\begin{eqnarray}
\Delta u = 6(1-e^u)\nonumber
\end{eqnarray}
on standard $S^2$, then $u\equiv0$.
\end{proposition}
\begin{proof}
Without of loss of generality we can assume $u$ axially symmetric about the $x_1$ axis,
take second order spherical harmonic $F=3x_1^2-1$ in (\ref{second eigenvalue}), then we have
\begin{eqnarray}
-(\Delta F)=6F, \quad \nabla^2 F=-6x_1^2g+6dx_1\otimes dx_1,
\end{eqnarray}that
\begin{eqnarray}
2\nabla^2 F-(\Delta F)g=6(x_1^2-1)g+12dx_1\otimes dx_1,
\end{eqnarray}and that\begin{eqnarray}
(2\nabla^2 F-(\Delta F)g)(\nabla u,\nabla u)=6(x_1^2-1)|\nabla u|^2+12|\nabla x_1\cdot\nabla u|^2.
\end{eqnarray}Since $u$ axially symmetric about the $x_1$ axis, $u$ is a function of $x_1$ and $\nabla u$ is a multiple of $\nabla x_1$, which implies $|\nabla x_1\cdot\nabla u|^2=|\nabla x_1|^2|\nabla u|^2=(1-x_1^2)|\nabla u|^2 $ and
\begin{eqnarray}
(2\nabla^2 F-(\Delta F)g)(\nabla u,\nabla u)=6(1-x_1^2)|\nabla u|^2.
\end{eqnarray}Inserting this into (\ref{second eigenvalue}) and taking $\l=6,$ we have
\begin{eqnarray}
\int_{S^2}3(1-x_1^2)|\nabla u|^2=0,
\end{eqnarray}
so $\nabla u\equiv0,\ \Delta u\equiv0$ and $u\equiv0$.
\end{proof}

\begin{rem}
Proposition \ref{axis sym} is nontrivial since there are nonzero axially symmetric solutions for $4<\lambda\leq 8$ and $\l\neq 6 $.
We refer the readers to Section 6 for more details.
\end{rem}
Let $X_1=x_2\partial_3-x_3\partial_2,\ X_2=x_3\partial_1-x_1\partial_3,\ X_3=x_1\partial_2-x_2\partial_1,$ be killing vector fields of $S^2$, then we have\begin{lem}\label{2.4}Assume $u$ a solution of \eqref{eq:2} then
$$\int_{S^2}X_iuX_juX_kue^u=0,\ \ \text{for}\ \ i,j,k\in\{1,2,3\}.$$\end{lem} \begin{proof}
Since killing vector fields commute with the Laplace operator we have
\begin{eqnarray}\label{eq:3}
\frac{\a}{2}\Delta X_iu +e^uX_iu=0,
\end{eqnarray}
and that
\begin{eqnarray}
\frac{\a}{2}\Delta X_iX_ju = -X_i( e^uX_ju)=- e^uX_iuX_ju- e^uX_iX_ju,
\end{eqnarray}
which implies that\begin{eqnarray*}
\int_{S^2}X_iuX_juX_kue^u=\int_{S^2}\left(\frac{\a}{2}\Delta X_iX_ju +e^uX_iuX_ju\right)X_ku\\
=\int_{S^2}X_iX_ju\left(\frac{\a}{2}\Delta X_ku +e^uX_ku\right)=\int_{S^2}X_iX_ju\cdot0=0,
\end{eqnarray*}this completes the proof.
\end{proof}

\section{Solutions of \eqref{eq:2}}
We assume $0<\a<1$ and $u$ a solution of \eqref{eq:2}, then $u$ satisfies the condition $\int_{S^2}e^ud\omega=1. $ When we write $ \int_{S^2}f,$ the volume element is induced from the metric
$g_0$, since the volume of $S^2$ is $4\pi,$ we have $ \int_{S^2}f=4\pi\int_{S^2}fd\omega$ and $\int_{S^2}e^u=4\pi. $
Following \cite{Lin1} let $ \Pi$ be the stereographic projection $ S^2\to\RR^2,$ with respect to the north
pole $N=(0,0,1):$$$\Pi(x_1,x_2,x_3):=\left(\frac{x_1}{1-x_3},\frac{x_1}{1-x_3}\right).$$
Suppose $u$ is a solution of \eqref{eq:2}, and let $$\overline{u}(y)=u(\Pi^{-1}(y))\ \ \text{for}\ \ y\in\RR^2.$$
Then $\overline{u} $ satisfies $$\Delta \overline{u}+\frac{8(e^{\overline{u}}-1)}{\a(1+|y|^2)^2}=0\ \ \text{in}\ \ \RR^2.$$
Now if we let $$v=\overline{u}-2\ln(1+|y|^2)+\ln(8/\a),$$ then $v$ satisfies \begin{equation}\label{eq:5}\Delta v+e^v=\frac{8(1-\a)}{\a(1+|y|^2)^2}>0\ \ \text{in}\ \ \RR^2,\end{equation} and $$\int_{\RR^2}e^vdy=\frac{8\pi}{\a},$$ more generally for a domain $ \Omega\subset\RR^2$ we have \begin{equation}\label{eq:8}\int_{\Omega}e^vdy=\frac{2}{\a}\int_{\Pi^{-1}(\Omega)}e^u
=\frac{8\pi}{\a}\int_{\Pi^{-1}(\Omega)}e^ud\omega,\end{equation} in particular let $B(0,r)=\{y\in\RR^2:|y|<r\}$ and assume $u$ even, then we have
\begin{equation}\label{eq:10}\int_{B(0,1)}e^vdy=\frac{8\pi}{\a}\int_{S^2\cap\{x_3<0\}}e^ud\omega
=\frac{4\pi}{\a}\int_{S^2}e^ud\omega=\frac{4\pi}{\a}.\end{equation}
We first prove that if $u$ is evenly symmetric about three orthogonal planes passing through the origin,
then $u$ is axially symmetric.
\begin{lem}\label{3.1}Let $ \frac{1}{4}\leq\a<1$ and $u$ be a solution of \eqref{eq:2}. If $$u(x_1,x_2,x_3)=u(|x_1|,|x_2|,|x_3|)\ \
\text{for}\ \ (x_1,x_2,x_3)\in S^2,$$ then $ X_iu\equiv0$ for some $i\in\{1,2,3\}$, thus $u$ is axially symmetric.\end{lem}
\begin{proof}Notice that the condition is equivalent to \begin{equation}\label{eq:9}u(\varepsilon_1x_1,\varepsilon_2x_2,\varepsilon_3x_3)= u(x_1,x_2,x_3) \end{equation} for $ \varepsilon_1,\varepsilon_2,\varepsilon_3\in\{\pm1\},$ and this implies that $$X_iu(\varepsilon_1x_1,\varepsilon_2x_2,\varepsilon_3x_3)=\varepsilon_i\varepsilon_1\varepsilon_2\varepsilon_3X_iu(x_1,x_2,x_3), $$ that $X_iu=0$ on $\{x_j=0\}\cap S^2$ and that \begin{equation}\label{eq:6}(X_1uX_2uX_3u)(\varepsilon_1x_1,\varepsilon_2x_2,\varepsilon_3x_3)=(X_1uX_2uX_3u)(x_1,x_2,x_3), \end{equation} for $\{i,j\}\subset\{1,2,3\},\ \varepsilon_i\in\{\pm1\}.$ Now suppose that $X_iu$ is not identically zero for every $i\in\{1,2,3\}$, let $ \varphi_i=(X_iu)\circ \Pi^{-1},$ then $\varphi_i$ is not identically zero and by \eqref{eq:3} we have
\begin{equation}\label{eq:4}\Delta \varphi_i=-\frac{8e^{\overline{u}}\varphi_i}{\a(1+|y|^2)^2}=-e^v\varphi_i\ \ \text{in}\ \ \RR^2.\end{equation} Since $X_1u=0$ on $(\{x_2=0\}\cup\{x_3=0\})\cap S^2$, we have $\varphi_1=0 $ on $\{y_2=0\}\cup\{|y|=1\},$ in particular $\varphi_1=0 $ on $\partial\Omega$ for $ \Omega=B(0,1)\cap\{y_2<0\}$. Using \eqref{eq:6} and \eqref{eq:8} we have \begin{equation}\label{eq:7}\int_{\Omega}e^vdy=\frac{8\pi}{\a}\int_{S^2\cap\{x_3<0,x_2<0\}}e^ud\omega
=\frac{2\pi}{\a}\int_{S^2}e^ud\omega=\frac{2\pi}{\a}\leq 8\pi,\end{equation} If $ \Omega\nsubseteq\{\varphi_1\neq0\}$ then the nodal line of $\varphi_1 $ divides $ \Omega$ into at least two regions $\Omega_1,\Omega_2 $. By \eqref{eq:4} we have $\l_{1,v}(\Omega_i)\leq0,$ by \eqref{eq:5}, \eqref{eq:7} and Lemma \ref{1.3} we have $$\int_{\Omega_i}e^vdy>4\pi\ \ \text{for}\ \ i=1,2,$$ and that $$\int_{\Omega}e^vdy\geq\int_{\Omega_1}e^vdy+\int_{\Omega_2}e^vdy>8\pi,$$
 which is a contradiction. Therefore $\varphi_1\neq 0 $ in $ \Omega,$ $\varphi_1$ does not change sign in $ \Omega,\ i.e.\ \exists\ \varepsilon_1'\in\{\pm1\}\ s.t.\ \varepsilon_1'\varphi_1>0$ in $ \Omega,$ thus $\varepsilon_1X_1u>0$ in $ \Sigma=S^2\cap\{x_1<0,x_2<0,x_3<0\}.$ Similarly we can find $\varepsilon_2',\varepsilon_3'\in\{\pm1\} $ such that $\varepsilon_2'X_2u>0,\ \varepsilon_3'X_3u>0$ in $ \Sigma.$ By \eqref{eq:6} we conclude that $\varepsilon_1'\varepsilon_2'\varepsilon_3'X_1uX_2uX_3u>0 $ in $S^2\setminus\{x_1x_2x_3=0\}$ and that $$\varepsilon_1'\varepsilon_2'\varepsilon_3'\int_{S^2}X_1uX_2uX_3ue^u>0,$$ which is contradict to Lemma \ref{2.4}. Therefore one of $X_iu$ is identically zero, this completes the proof.\end{proof}
 We need the following result to prove the even symmetry about a plane.
 \begin{lem}\label{3.2}Let $ \frac{1}{4}\leq\a<1$ and $u$ be a solution of \eqref{eq:2}, $u(x)=u(-x)$. If $P\in S^2$ is a critical point of $u$, then $u$ is evenly symmetric
about the plane passing through the origin $O$ and orthogonal to $ \overrightarrow{OP}$.\end{lem}
\begin{proof}Without of loss of generality we can assume $P=(0,0,1),$ it is enough to
prove that $u$ is symmetric about the $x_1x_2$-plane. Define $u^*(x_1,x_2,x_3)=\\u(x_1,x_2,-x_3)$ and $ \widetilde{u}(x)=u(x)-u^*(x).$ Notice that $ \widetilde{u}(x_1,x_2,0)=0$, for all $(x_1,x_2,0)\in S^2. $ Then $ \widetilde{u}$ satisfies
\begin{equation}\label{eq:3.7}
\frac{\a}{2}\Delta \widetilde{u}+c(x)\widetilde{u}=0,\ \ \text{on}\ \ S^2,
\end{equation} where $$c(x):=\int_0^1e^{tu+(1-t)u^*}dt.$$ Since $u(x)=u(-x)$, we have $u^*(x_1,x_2,x_3)=u(-x_1,-x_2,x_3)$ and $ \\-\widetilde{u}(x_1,x_2,x_3)=\widetilde{u}(-x_1,-x_2,x_3),\ \widetilde{u}(P)=0. $ Since $P\in S^2$ is a critical point of $u$, we have $\nabla u^*(P)=\nabla u(P)=0,\ \nabla\widetilde{u}(P)=0.$

 Now let $\widetilde{v}=\widetilde{u}\circ \Pi^{-1} $
 then $\widetilde{v}$ satisfies \begin{equation}\label{eq:3.8}\Delta \widetilde{v}=-\frac{8c\circ \Pi^{-1}\widetilde{v}}{\a(1+|y|^2)^2}:=-\widetilde{c}\widetilde{v}\ \ \text{in}\ \ \RR^2,\end{equation}
 and $\widetilde{v}=0 $ on $ \partial B(0,1),\ \widetilde{v}(0)=\nabla\widetilde{v}(0)=0,\ -\widetilde{v}(y)=\widetilde{v}(-y)$ for $y\in\RR^2.$
If $\widetilde{v}\not\equiv0 $ then $\widetilde{v}(y)=Q(y)+O(|y|^{m+1}) $ as $y\to0,$ where $Q(y)$ is a quadratic polynomial of degree $m$ with $m\geq 2,$ in fact $Q(y)=Re(a(y_1+iy_2)^m)$ for some $0\neq a\in\CC.$ Since $-\widetilde{v}(y)=\widetilde{v}(-y)$, we have $m$ odd and $m\geq 3.$ Thus, the nodal line $\{y:\widetilde{v}(y)=0\}$ divides $B(0,r)$ into $2m$ regions for some $0<r<1,$ we can order them in counter-clockwise direction as $ V_1,\cdots,V_{2m}, $ then $V_{j+m}=-V_j $ for $1\leq j\leq m,$ and we can assume
$ (-1)^j\widetilde{v}(y)>0$ in $V_j $. Now we claim that $\{\widetilde{v}\neq 0\}\cap B(0,1) $ has at least four simply connected components $ \Omega_1, \Omega_2,\ \Omega_3=-\Omega_1,\ \Omega_4=\Omega_2.$ Once the claim is true, define $v_1,v_2$ as follows $$v_1(y)={u}(\Pi(y))-2\ln(1+|y|^2)+\ln(8/\a),$$ and $$v_2(y)={u}^*(\Pi(y))-2\ln(1+|y|^2)+\ln(8/\a),$$ then $v_1$ and $v_2$ both satisfy \eqref{eq:5} and $\widetilde{v}=v_1-v_2, $ $v_1=v_2$ on $ \partial\Omega_j,j=1,2,$ $v_1(y)=v_2(-y).$ Applying the Sphere Covering Inequality (Lemma \ref{1.2}) on $ \Omega_j,j=1,2$ and using \eqref{eq:10} we obtain that \begin{align*}\frac{4\pi}{\a}=\int_{B(0,1)}e^{v_1}dy\geq\sum_{j=1}^4\int_{\Omega_j}e^{v_1}dy
=\sum_{j=1}^2\int_{\Omega_j}(e^{v_1}+e^{v_1})dy>\sum_{j=1}^24\pi=8\pi.\end{align*} Hence $\a<\frac{1}{4},$ which is a contradiction. Thus $\widetilde{v}\equiv0,\ \widetilde{u}\equiv0,\ u\equiv u^* $ and the result is true. It remains to prove the claim.

Let $ \gamma$ be the connected component of $ \{\widetilde{v}=0\}$ containing $0$. Since $-\widetilde{v}(y)=\widetilde{v}(-y), $ we have $\gamma=-\gamma.$ If $ \gamma$ is bounded, let $V_0$ be the unbounded component of $\RR^2\setminus\gamma$, then $V_0=-V_0$ and $ \partial V_0\subseteq\gamma.$ There exists a neighborhood of $ \partial V_0$ denoted by $U_0$ such that $U_0=-U_0$ and $ \widetilde{v}$ does not change sign in $ U_0\cap V_0,$ which is contradict to $-\widetilde{v}(y)=\widetilde{v}(-y).$ Therefore $ \gamma$ is unbounded, $ \gamma$ must intersect $ \partial B(0,1),$ since $ \partial B(0,1)\subseteq\{\widetilde{v}=0\},$ we have $ \partial B(0,1)\subseteq\gamma.$

Let $V_j^*$ be the component of $\RR^2\setminus\gamma$ containing $V_j$, then $V_j^*$ is simply connected and $V_j\subseteq V_j^*\subseteq B(0,1) $ for $1\leq j\leq 2m,$ as $-\widetilde{v}(y)=\widetilde{v}(-y), $ we have $V_{j+m}^*=-V_j^* $ for $1\leq j\leq m.$  Define $ A=\{V_{2j-1}^*:j\in\ZZ\cap[1,m]\},\ B=\{V_{2j}'^*:j\in\ZZ\cap[1,m]\},$ then the sets $A$ and $B$ are finite and nonempty, $B=\{-V:V\in A\}.$ If $V=V_{2j-1}^*=V_{2k}^*\in A\cap B$ then $ 0\in\partial V\subseteq\gamma,$ and there exists a neighborhood of $ \partial V$ denoted by $U$ such that $ \widetilde{v}$ does not change sign in $ U\cap V,$ which is a contradiction since $ \widetilde{v}(y)$ has different signs in $V_{2j-1}$ and $V_{2k},$ therefore $ A\cap B=\emptyset.$

If $A$ has only one element, so has $B$, and $V_1^*=V_3^*\neq V_2^*=V_4^*. $ Choose $p_j\in V_j^*\cap \partial B(0,r)$ for $1\leq j\leq 2m,$ then we can find curves $ \gamma_j$ in $V_j^*\setminus B(0,r)$ connecting $p_j$ and $p_{j+2}$ for $j=1,2,$ and $ \gamma_0$ in $\overline{V_1}\cup\overline{V_2}\cup\overline{V_3}$ connecting $p_1$ and $p_{3},$ such that $ \gamma_0\cap\partial B(0,r)=\gamma_1\cap\partial B(0,r)=\{p_1,p_3\},\ \gamma_2\cap\partial B(0,r)=\{p_2,p_4\}$ and $ \gamma_*=\gamma_0\cap\gamma_1$ is an embedded circle. Then $p_2$ and $p_{4}$ are separated by $ \gamma_*$ and $ \gamma_2\cap\gamma_0= \gamma_2\cap\gamma_1=\emptyset$, which is a contradiction.

Therefore $A$ has more than one elements, say $ V_k^*,V_l^*\in A,\ V_k^*\neq V_l^*.$ Choose $k',l'\in[1,2m]$ such that $|k-k'|=m,\ |l-l'|=m,$ then $ V_k^*,V_l^*,V_{k'}^*=-V_k^*, V_{l'}^*=-V_l^*$ are distinct. Let $W_1=V_k^*,W_2=V_l^*,W_3=V_{k'}^*=-W_1, W_4=V_{l'}^*=-W_2,$ then $W_i\cap W_j=\emptyset$ for $1\leq i<j\leq4.$ Now $W_i$ contains at least one simply connected component of $\{\widetilde{v}\neq 0\}\cap B(0,1) $, say $ \Omega_i$, and $ -\Omega_i\subseteq W_{i+2}$ for $i=0,1,$ thus $ \Omega_1,\Omega_2$ satisfy the condition in the claim. This completes the proof.\end{proof}

Now we can prove Theorem \ref{1.1}.

\begin{proof}[Proof of Theorem \ref{1.1}] Assume that $P\in S^2$ is a maximum point of $u$. Without of loss of generality we can assume $P=(1,0,0),$ by Lemma \ref{3.2} we have $u(x_1,x_2,x_3)=u(-x_1,x_2,x_3). $ Let $u_0$ be the restriction of $u$ to $S^2\cap\{x_1=0\}$ and
assume that $Q\in S^2$ is a maximum point of $u_0$. Since $u$ is symmetric about $x_2x_3$-plane, $Q$ is
also a critical point of $u$ on $S^2.$ Without of loss of generality we can assume $Q=(0,1,0),$ by Lemma \ref{3.2} we have $u(x_1,x_2,x_3)=u(x_1,-x_2,x_3). $ Now we have $u(x_1,x_2,x_3)=u(|x_1|,|x_2|,x_3). $ Since $u(x)=u(-x)$, we also have $u(x_1,x_2,x_3)=u(-x_1,-x_2,-x_3)=u(|-x_1|,|-x_2|,-x_3)=u(|x_1|,|x_2|,x_3) $ and that $u(x_1,x_2,x_3)=u(|x_1|,|x_2|,|x_3|). $ By Lemma \ref{3.1}, $u$ is axially symmetric. If $\a=\frac{1}{3}$ then $\l=6,$ by Proposition \ref{axis sym} we have $u=0.$ This completes the proof.\end{proof}
Now we consider solutions of \eqref{eq:2} with no symmetry assumption.
 \begin{lem}\label{3.3}Let $ \frac{1}{4}\leq\a<1$ and $u$ be a solution of \eqref{eq:2}. If $e^u\leq 6\a$ in $S^2$, then 
 $u$ is axially symmetric. In particular, if $\a=\frac{1}{3}$ and $e^u\leq 2$ then $u\equiv0.$\end{lem}
\begin{proof}Define $u^*(x)=u(-x)$ and $ \widetilde{u}(x)=u(x)-u^*(x).$ Then $ \widetilde{u}$ satisfies
\begin{equation}\label{eq:4.7}
\frac{\a}{2}\Delta \widetilde{u}+c(x)\widetilde{u}=0,\ \ \text{on}\ \ S^2,
\end{equation} where $$c(x):=\int_0^1e^{tu+(1-t)u^*}dt.$$ Notice that the eigenvalues of $-\Delta$ on $S^2$ are $\l_j=j(j+1)$ for $j\in \ZZ$, the eigenspaces $E_j=\ker(\Delta+\l_j)$ ($j\in \ZZ,\ j\geq0$) are orthogonal and complete in $L^2(S^2)$, $E_j$ is the space of $j$th spherical harmonics, in particular $E_0=\text{span}\{1\},\ E_1=\text{span}\{x_1,x_2,x_3\}$, $f(x)=f(-x)$ for $f\in E_j,\ j$ even  and $-f(x)=f(-x)$ for $f\in E_j,\ j$ odd. Let $P_j$ be the orthogonal projection from $L^2(S^2)$ to $E_j$ then we have $ P_ju^*=(-1)^jP_ju$ and $ P_j\widetilde{u}=0$ for $j$ even. Since $\l=2/\a>2,$ by Lemma \ref{4.2} we have $\int_{S^2}x_ie^u=\int_{S^2}x_i=0$ and
$$\int_{S^2}-2x_iu=\int_{S^2}\Delta x_iu=\int_{S^2}x_i\Delta u=\l\int_{S^2}x_i(1-e^u)=0,$$ which implies $P_1u=0$ and $P_1u^*=0,\ P_1\widetilde{u}=0.$ Now we have $ P_j\widetilde{u}=0$ for $j=0,1,2$ and \eqref{eq:4.7} implies that
\begin{align*}\int_{S^2}c(x)\widetilde{u}^2&=-\frac{\a}{2}\int_{S^2}\widetilde{u}\Delta \widetilde{u}=\frac{\a}{2}\sum_{j=0}^{+\infty}\l_j\|P_ju\|_{L^2(S^2)}^2
=\frac{\a}{2}\sum_{j=3}^{+\infty}\l_j\|P_ju\|_{L^2(S^2)}^2\\ &\geq\frac{\a}{2}\sum_{j=3}^{+\infty}\l_3\|P_ju\|_{L^2(S^2)}^2
=\frac{\a}{2}\l_3\sum_{j=0}^{+\infty}\|P_ju\|_{L^2(S^2)}^2=6\a\int_{S^2}\widetilde{u}^2.\end{align*}On the other hand since $e^u\leq 6\a$ in $S^2$ we have $e^{u^*}\leq 6\a$ in $S^2$ and $e^{tu+(1-t)u^*}\leq 6\a$ in $S^2$ for $0\leq t\leq1,$ thus $c(x)\leq6\a$ in $S^2$. Now we must have $c(x)\widetilde{u}^2=6\a\widetilde{u}^2 $ in $S^2$, if $\widetilde{u}(x)\neq0 $ for some $x\in S^2$ then we have $c(x)=6\a$ and $e^{tu(x)+(1-t)u^*(x)}= 6\a$ for $0\leq t\leq1,$ therefore $e^{u(x)}=e^{u^*(x)}= 6\a$, which is a contradiction. Thus $\widetilde{u}\equiv0 $ in $S^2,$ and $u(x)=u^*(x)=u(-x),$ by theorem \ref{1.1} we know that $u$ is axially symmetric. If $\a=\frac{1}{3}$ and $e^u\leq 2$ then $e^{u^*}\leq 6\a$ and $u(x)=u^*(x)=u(-x),$ by theorem \ref{1.1} we know that $u\equiv0.$ This completes the proof. \end{proof}

\section{Rigidity of Hawking mass for surface with even symmetry}

Now we can prove Theorem \ref{1.2}-\ref{1.5} as in \cite{Sun}, we still need the following lemmas in the proof. The following rigidity result is some kind of  positive mass theorem in compact case(see \cite{Miao}, \cite{ST}, and \cite{HW}).
\begin{lem}\label{sphere rigidity}
Let (M, g) be a compact orientable Riemannian 3-manifold with scalar curvature $R(g)\geq0$ and $\partial M$ isometric to round $S^2$ with mean curvature $H=2$. Then (M, g) is isometric to the unit ball in ($\mathbb{R}^3$, $\delta$).
\end{lem}
To prove Theorem 1 we need a rigidity result for hyperbolic case of sphere, see Theorem 3.8 in \cite{ST07} by the first author and L. F. Tam.
\begin{lem}\label{hyperbolic sphere rigidity}
Let (M, g) be a compact orientable Riemannian 3-manifold with scalar curvature $R(g)\geq-6$ and $\partial M$ isometric to round $S^2$ with mean curvature $H=2\sqrt{2}$. Then (M, g) is isometric to the unit ball in hyperbolic space $\mathbb{H}^3$.
\end{lem}
After Theorem 1, Lemma \ref{sphere rigidity} and Lemma \ref{hyperbolic sphere rigidity}, now we are in the position to prove Theorem \ref{1.2} and Theorem \ref{1.3}.

\begin{proof}[Proof of Theorem \ref{1.2} and Theorem \ref{1.3}]
If $m_H(\Sigma)= 0$ on stable $CMC$ surface $\Sigma$, without loss of generality, assume $|\Sigma|=4\pi$, then $H=2(resp. H=2\sqrt{2})$.
From the proof in \cite{Sun} we know that there exists a conformal map $ \varphi:\Sigma\to S^2\subseteq\RR^3$ such that $ \int_{\Sigma}\varphi=0.$ Denote the metric on $ \Sigma$ as $g=\varphi^*(e^ug_0)$ then $u$ satisfies $$\Delta u=6(1-e^u)\ \ \text{on}\ \ S^2.$$ Let $ \rho:\Sigma\to\Sigma$ be the isometry such that $ \rho^2=id$ and $ \rho(x)\neq x$ for $x\in\Sigma,$ then $ \widetilde{\rho}=\varphi\circ\rho\circ\varphi^{-1}:S^2\to S^2$ is a conformal map. Now we have \begin{align*}\varphi^*(e^ug_0)=g=\rho^*g=\rho^*\varphi^*(e^ug_0)=(\varphi\circ\rho)^*(e^ug_0)
\\=(\widetilde{\rho}\circ\varphi)^*(e^ug_0)=\varphi^*\widetilde{\rho}^*(e^ug_0),\end{align*} thus $(e^ug_0)=\widetilde{\rho}^*(e^ug_0). $ Assume $\widetilde{\rho}^*g_0=e^vg_0 $ then $u=u\circ\widetilde{\rho}+v.$ Since $\widetilde{\rho}^*g_0 $ has constant Gaussian curvature $1$, $v$ satisfies $ \Delta v=2(1-e^v).$ Now we have \begin{align*}6(1-e^u)=\Delta u=\Delta(u\circ\widetilde{\rho})+\Delta v=(\Delta u)\circ\widetilde{\rho}e^v+\Delta v=6(1-e^u)\circ\widetilde{\rho}e^v+\Delta v
\\=6(e^v-e^{u\circ\widetilde{\rho}+v})+\Delta v=6(e^v-e^{u})+2(1-e^v)=6(1-e^{u})-4(1-e^v),\end{align*} thus $e^v=1 , \ u=u\circ\widetilde{\rho}$ and $\widetilde{\rho}$ is an isometry of $S^2$. Now there exists an orthogonal matrix $A\in O(3)$ such that $\widetilde{\rho}x=Ax. $ Since $ \rho^2=id$, we have $ \widetilde{\rho}^2=id$ and $A^2=I_3,$ the eigenvalues of $A$ must be $\pm1.$ Since $ \rho(x)\neq x$ for $x\in\Sigma,$ we have $ \widetilde{\rho}(x)\neq x$ for $x\in S^2,$ and $1$ is not an eigenvalue of $A.$ Therefore the eigenvalues of $A$ must be $-1,$ which implies $A=-I_3$ and $\widetilde{\rho}(x)=-x,\ u(x)=u(-x). $
By Theorem \ref{1.1} we know $u=0$ and $\Sigma$ is standard $S^2$ in $\mathbb{R}^3$. Then by Lemma \ref{sphere rigidity}(resp. Lemma \ref{hyperbolic sphere rigidity}), we conclude that $\Omega$ isometric to unit ball in $\mathbb{R}^3(\mathbb{H}^3_{-1})$.
\end{proof}
\begin{rem}Since the Gaussian curvature of
 $ \Sigma$ is $K_{\Sigma}=3-2e^{-u}$ for $|\Sigma|=4\pi$, by Lemma \ref{3.3} we know that if the even symmetry assumption in Theorem
 \ref{1.2}-\ref{1.5} is replaced by $K_{\Sigma}{\dfrac{|\Sigma|}{4\pi}}\leq2, $
 then the result is still true.\end{rem}

Theorem \ref{1.2}(resp. Theorem \ref{1.3})can help us to understand Willmore functional in manifold with scalar curvature $R(g)\geq0(resp. R(g)\geq-6)$.
\begin{cor}
Let (M,g) be a complete Riemnnian three manifold with scalar curvature $R(g)\geq0(resp. R(g)\geq-6)$, $\Sigma=\partial \Omega$ is a stable CMC sphere, then $W(\Sigma)\leq 4\pi$.
\\If $\Sigma$ has even symmetry, then equality holds if and only if $\Sigma$ is standard $S^2$ and $\Omega$ isometric to unit ball in $\mathbb{R}^3(resp. \mathbb{H}^3)$.
\end{cor}

In order to prove Theorem \ref{1.4} we need the following isoperimetric inequality  of \cite{Shi} which also plays a key role in proving the existence of isoperimetric surface for all volume in $AF$ three manifold. It says that if there exists a Euclidean ball in an AF manifold with nonnegative scalar curvature, then the AF manifold must be $\mathbb{R}^3$.
\begin{lem}\cite{Shi} \label{shi rigidity}
Suppose (M, g) is an AF manifold with scalar curvature $R(g)\geq0$.
Then for any $V>0$ \begin{eqnarray}I(V)\leq(36\pi)^{\frac{1}{3}}V^{\frac{2}{3}}.\end{eqnarray}
There is a $V_0 > 0$ with
\begin{eqnarray}I(V_0)=(36\pi)^{\frac{1}{3}}V_0^{\frac{2}{3}}\end{eqnarray}
if and only if (M, g) is isometric to $\mathbb{R}^3$.
\end{lem}
Also there has an analogous result for isoperimetric profile on AH manifold, see Propostion 3.3 in \cite{JSZ}.
\begin{lem}\cite{JSZ} \label{hyperbolic shi rigidity}
Suppose (M, g) is an AH manifold with scalar curvature $R(g)\geq-6$.
Then for any $V>0$ \begin{eqnarray}I(V)\leq I_{\mathbb{H}}(V).\end{eqnarray}
There is a $V_0 > 0$ with
\begin{eqnarray}I(V_0)=I_{\mathbb{H}}(V_0)\end{eqnarray}
if and only if (M, g) is isometric to $(\mathbb{H}^3, g_{\mathbb{H}}) $.
\end{lem}

Now we can prove the rigidity of isoperimetric surface with even symmetry:
\begin{proof}[Proof of Theorem \ref{1.4} and Theorem \ref{1.5}]
If there is an isoperimetric surface  $\Sigma$ with even symmetry and $m_H(\Sigma)= 0$, assume $|\Sigma|=4\pi$, then $H=2$. By Theorem \ref{1.2}, the isoperimetric region is a Euclidean ball of volume $\frac{4}{3}\pi$. So we have
\begin{eqnarray}
4\pi=I(\frac{4}{3}\pi)=(36\pi)^{\frac{1}{3}}(\frac{4}{3}\pi)^{\frac{2}{3}},
\end{eqnarray}
by the rigidity part of Lemma \ref{shi rigidity}, we conclude that $(M, g)$ is isometric to $\mathbb{R}^3$.
\\Theorem \ref{1.5} follows similarly by Theorem \ref{1.3} and Lemma \ref{hyperbolic shi rigidity}.
\end{proof}

\section{A final remark}

In this section, we give a sketched proof for the existence of nonzero axially symmetric solutions of \eqref{eq:2} for $\frac{1}{4}\le \a < \frac{1}{2}$ and
$\a\not= \frac{1}{3}$. We will also discuss the uniqueness of these solutions.

We will adopt some notations in \cite{Lin}. Let $0<\a<1$ and $u$ a solution of \eqref{eq:2}, we write $\l=2/\a=l+2$ and
$$\widetilde{u}(y)\,=\,{u}(\Pi^{-1}(y))-\l\ln(1+|y|^2)\,+\,\ln(8/\a),$$
where $\Pi^{-1} $ is defined in Section 4. Then $\widetilde{u} $ satisfies
\begin{eqnarray}\label{3.16}
\Delta\widetilde{u}\,+\,(1+|y|^2)^le^{\widetilde{u}}\,=\,0,~~\text{in}\  \ \RR^2
\end{eqnarray}
and
\begin{eqnarray}
\int_{\RR^2}\,(1+|y|^2)^{l}\,e^{\widetilde{u}(y)}\,dy\,=\,4\pi\l.\end{eqnarray}
Let $w(r;s) $ be the unique solution of
\begin{eqnarray}\label{3.18}
\left\{\begin{array}{l}w''+\dfrac{w'}{r}+(1+r^2)^le^{w}\,=\,0,\\ w(0;s)=s,\ \ w'(0;s)=0,
\end{array}\right.
\end{eqnarray}
then $w(r;s) $ exists for all $r>0$ and always satisfies the asymptotic
behavior
\begin{eqnarray}
u(r;s)\,=\,-\b(s)\ln r\,+\,O(1)\ \ \text{at}\ \ \infty.
\end{eqnarray}
Integrating \eqref{3.16}, we can compute $\b(s)$ through the following formula:
\begin{eqnarray}2\pi\b(s)\,=\,\int_{\RR^2}\,(1+|y|^2)^{l}\,e^{w(|y|;s)}\,dy.\end{eqnarray}
Now the axially symmetric solutions about the $x_3$ axis of \eqref{eq:2} are of the form
$${u}(\Pi^{-1}(y))\,-\,\l\ln(1+|y|^2)\,+\,\ln(8/\a)\,=\,w(|y|;s),\ \ \b(s)\,=\,2\l,$$
and the number of such solutions equals to that of solutions of $\b(s)\,=\,2\l.$
We know from Lemma 4.1 in \cite{Lin} that if $l>1$,
then
$$\lim\limits_{s\to+\infty}\,\b(s)\,=\,4l~~~{\rm and}~~~\lim\limits_{s\to-\infty}\,\b(s)\,=\,2(2+2l).$$
These imply that for $\l>4$, we have $l>2$ and
$$2\l\,<\,\lim\limits_{s\to+\infty}\,\b(s)\,<\,\lim\limits_{s\to-\infty}\,\b(s).$$
Set $\b_0=\min\limits_{s\in\RR}\b(s)$, if $\b_0<2\l$, then $\b(s)=2\l$ has at least two solutions, consequently,
\eqref{eq:2} has at least two solutions which are axially symmetric about the $x_3$ axis.
On the other hand, the solution of \eqref{3.18} given by
$$w(r;s)\,=\,-\l\ln(1+|y|^2)\,+\,\ln(8/\a),\ s\,=\,\ln(8/\a)\,=\,\ln(4\l)$$
corresponds to the trivial solution $u\equiv 0$ on $S^2$, therefore,
we have $\b(s)=2\l$ for $s=\ln(4l)$, if $\dot{\b}(s)=0$ for $s=\ln(4l)$, then the linearized equation of \eqref{eq:2} at $u\equiv0$ must be singular and $\l$ is an eigenvalue of $-\Delta$ on $S^2$, $\l=k(k+1)$ for a nonnegative integer $k$. Therefore, if $4<\l\leq8$ and $\l\neq6$, then $\dot{\b}(s)\neq0$ for $s=\ln(4l)$ and $\ln(4l)$ is not the minimum point of $\b(s)$, $\b_0<\b(\ln(4l))=2\l$ and \eqref{eq:2} has nonzero axially symmetric solutions.

We can also prove the uniqueness of nonzero axially symmetric solutions about the $x_3$ axis for $0<|\l-6|<\d_0$ for a sufficiently small $\d_0$. In fact,we have $\ddot{\b}(s)>0$ for $s=\ln(4l),\ \l=6,$ and $\ddot{\b}(s)$ is continuous with respect to $s$ and $\l,$ which implies $\ddot{\b}(s)>0$ for $|s-\ln(4l)|<\d_1,|\l-6|<\d_1,$ here $\d_1>0$ is a constant, therefore $\b(s)=2\l,\ |s-\ln(4l)|<\d_1,$ has at most two solutions for $|\l-6|<\d_1,$ here $\d_1>0$ is a constant. Now we only need to prove the existence of $0<\d_0<\d_1$ such that axially symmetric solutions of \eqref{eq:2} satisfy $|u|<\d_1$ for $|\l-6|<\d_0,$ if this is not true then there exists a sequence of axially symmetric solutions $u_i$ about the $x_3$ axis of \eqref{eq:2} with $\l=\l_i$ tending to $6$ such that $\|u_i\|_{L^{\infty}}\geq \d_1$, by using a result of Li in \cite{Li} and taking a subsequence if necessary, we may assume that $u_i$ converge uniformly to a solution of \eqref{eq:2} with $\l=6$ or blowing up at $k$ points. In the first case, by Proposition \ref{axis sym}, we have $u_i$ converge uniformly to 0, a contradiction. In the second case, we have $\l=6=2k,$ since $u_i$ are axially symmetric about the $x_3$ axis, the blow up points must lie in $\{(0,0,\pm1)\}$ and $k\leq2$, this leads to a contradiction. This will enable us to conclude the uniqueness of nonzero axially symmetric solutions about the $x_3$ axis for $0<|\l-6|<\d_0$.

It would be very interesting to study exactly how many non-zero solutions of \eqref{eq:2} for any given $\alpha$
(module rotations) and whether or not they are axially symmetric.

\section{Acknowledgements}
We are very grateful to Juncheng Wei for suggesting us the sphere covering inequality of Gui and Moradifam \cite{Gui}. The second author also want to thank Jie Qing for pointing out the relation between rigidity of Hawking mass and eigenvalue problem. Many thanks to Changfeng Gui for many discussions on mean field equation.

\end{document}